\documentclass[a4paper, 11pt]{article}
\usepackage[utf8]{inputenc}
\usepackage[T1]{fontenc}
\usepackage{lmodern}
\usepackage{graphicx}
\usepackage{subfigure}
\usepackage{amsmath, amsfonts, amssymb,amsthm}
\usepackage{hyperref}
\hypersetup{colorlinks = true, urlcolor = blue, bookmarksopen = true}
\newtheorem{theorem}{Theorem}

\newtheorem{proposition}{Proposition}
\newtheorem{definition}{Definition}
\newtheorem{lemma}{Lemma}

\newtheorem{corollary}{Corollary}

\begin{document}

\title{Adaptive confidence bands in the nonparametric fixed design regression model}

\author{Pierre-Yves Massé, William Meiniel \\
\'Ecole Normale Supérieure de Cachan}

\maketitle

\begin{abstract}
In this note, we consider the problem of existence of adaptive confidence bands in the fixed design regression model,
adapting ideas in Hoffmann and Nickl \cite{HN} to the present case. In the course of the proof, we show that
sup-norm adaptive estimators exist as well in regression.
\end{abstract}

\section{Introduction}

We observe random variables $Y_i$'s and assume that, for $n \in \mathbb{N}$,
\begin{equation}
\label{eqmd}
Y_i = f(x_i) + \varepsilon_i, \hspace{3 mm} 1 \leq i \leq n
\end{equation}
where $x_i = \frac{i}{n}$, $ \varepsilon_i \sim \mathcal{N}(0,\sigma^2)$ are independent and identically distributed random variables defined on $(\Omega, \mathcal{F},P)$,
and  $f$ is the unknown regression function. We further assume the variance $\sigma^2$ is known. Our aim is to
reconstruct $f$ from the sample. \newline

Let us fist define the parameter space. We assume $f$ belongs to some Hölder space.
The Hölder space $\mathcal{C}^t$ for $0<t\leq 1$ is the space of continuous functions $f$ on $[0,1]$ such that
\begin{eqnarray*}
\Vert f \Vert_{\mathcal{C}^t} = \Vert f \Vert_{\infty} + \sup_{x \neq y} \frac{|f(x)-f(y)| }{| x - y |^t} < \infty \hspace{3 mm} \text{if} \hspace{3 mm} 0<t<1 \\
\Vert f \Vert_{\mathcal{C}^t} = \sup_{x \neq y} \frac{|f(x+y)+f(x-y)-2f(x)|}{| y|} < \infty \hspace{3 mm} \text{if} \hspace{3 mm} t=1.
\end{eqnarray*}
If $t>1$, $\mathcal{C}^t$ is the space of functions such that the $\lfloor t \rfloor$-th derivative of $f$ exists, is continuous
and belongs to $\mathcal{C}^{t - \lfloor t \rfloor}$ if $t$ is not an integer, and such that the $t$-th derivative of $f$ exists, is continuous
and is in $\mathcal{C}^1$ otherwise. Define then $\Vert f \Vert_{\mathcal{C}^t} = \Vert f^{\lfloor t \rfloor} \Vert_{\mathcal{C}^{t - \lfloor t \rfloor}}$. Therefore, for all $t$, $\Vert f \Vert_{\mathcal{C}^t}$ is a norm of $f$.\newline

Working with this definition may prove difficult, and we thus use the wavelet basis characterisation of $\mathcal{C}^t$ :
whether or not $f$ belongs to it depends on the size of the coefficients of its decomposition over this basis.
We use Daubechies wavelets on the unit interval, as in \cite{HN}. Denote, accordingly, by $\phi$ and $\psi$ the scaling functions, $\phi_m = \phi(\cdot - m)$,
$\psi_{jm} = 2^{j/2}\psi(2^j \cdot - m)$ and, for $<\cdot,\cdot>$ the usual inner product in $L^2([0,1])$, 
\begin{equation*}
\Vert f \Vert_{t,\infty} = \max \left( \frac{}{} \sup_m |<\phi_m,f>| \, , \, \sup_{j,m} 2^{j(s+1/2)} |<\psi_{jm},f>| \right).
\end{equation*}

Now, Theorem~4.4 in \cite{Cohen} gives the equivalence : for $t>0$, $\mathcal{C}^t$ is the set of continuous functions on $[0,1]$ such that $\Vert f \Vert_{t,\infty} < \infty $. Moreover, $\Vert f \Vert_{t,\infty}$ is a norm equivalent to $\Vert f \Vert_{\mathcal{C}^t}$. \newline

For fixed $B > 0$, define
$$ \Sigma(t) = \left\lbrace f : \mathbb{R} \rightarrow \mathbb{R} \hspace{1 mm} \left\vert \frac{}{} \hspace{1 mm} ||f||_{t,\infty} \leq B \right. \right\rbrace.$$ \newline
One has typically no knowledge of the regularity of $f$. But estimators which achieve the optimal risk in sup-norm loss for whatever $t$ have been constructed. They
are called adaptive, and exist in the density and the white noise cases, as shown respectively
by Giné and Nickl \cite{GN2} and Goldenshluger and Lepski \cite{Gold}. Now, Brown and Low \cite{BrLw} have proven that the white noise and regression cases
are asymptotically equivalent, and therefore one may legitimately expect adaptive estimators to exist as well in the latter. For the sake of reference we prove it
in section~\ref{pth5}, that is, we prove the following theorem, where

\begin{equation*}
r_n(t) = \left( \frac{\log n}{n} \right)^{t/(2t+1)} \hspace{1,0 cm} t > 0.
\end{equation*}

\begin{theorem}
\label{thm5}
Let $Y_1,\dots,Y_n$ verify $Y_i = f(i/n) + \varepsilon_i$ where the $\varepsilon_i$'s are i.i.d. $\mathcal{N}(0,\sigma^2)$. \newline
Then, for every integer $l > 0$, there exists an estimator $\hat{f}_n(x) := \hat{f}_n(x,Y_1,\dots,Y_n,l)$
and an integer $n_0$ such that, for every $t$,
$0 < t \leq l$, some constant $D(B,l)$ and every $n \geq n_0$, we have $$\sup_{f \in \Sigma(t)} E_f \Vert \hat{f}_n - f \Vert_{\infty} \leq D(B,l) r_n(t).$$
\end{theorem}

We use the following notations. For any random variable $X : \Omega \rightarrow \mathbb{R}^n$ measurable with respect to $\sigma(Y_1, ..., Y_n)$ and every function $f \in \Sigma(r)$,
$E_f(X)$ is the expectation of $X$ when the function in Equation~\ref{eqmd} is $f$.
For all $f \in \Sigma(r)$, and all $A \in \mathcal{B}(\mathbb{R}^n)$, $P_f \left( \frac{}{} (Y_1, \dots, Y_n) \in A \right) = E_f[1_A (Y_1, \dots, Y_n)].$ \newline

Thanks to these estimators, one may then want to construct confidence bands for $f$, that is data-driven sets which
cover $f$ at all points simultaneously.

\begin{definition}
A confidence band is a family of random intervals $$C_n = C(Y_1,\dots,Y_n) = \left\lbrace  \left[c_n(y),c_n'(y)\right] \right\rbrace_{y \in [0,1]}.$$ 
Define the diameter of $C_n$, $|C_n| = \sup_{y \in [0,1]} |c_n'(y) - c_n(y)|$.
\end{definition} 

For instance, Claeskens and Van Keilegom~\cite{CVK} construct such bands for the function and its derivatives.
However, just as with estimators, one would like the diameter of the band to be optimal for all $t$, that is :

\begin{equation*}
\forall t > 0, \sup_{f \in \Sigma(t)} E_f|C_n| \leq L r_n(t)
\end{equation*}

\noindent for some constant $L$ independent of $t$. Here, we restrict ourselves to the case where $f$ belongs to $\Sigma(s)$ or $\Sigma(r)$ for $0 < r < s$ given.
Now, results by Low \cite{Low}, by Genovese and Wasserman \cite{GW}, and also our findings below, imply that to construct such sets is not possible for $\Sigma(s) \cup \Sigma(r)$. To circumvent this problem,
two paths have recently been considered. Both restrict the parameter space,
so that such negative results no longer apply, but at the same time the parameter space remains as large as possible.
The one studied by Giné and Nickl \cite{GN} in the density model consists in removing permanently
some functions from the model, according to a self-similarity condition. Bull \cite{Bull} then extends their results,
in particular to the white noise model. Hoffmann and Nickl \cite{HN} on the other hand let the parameter space evolve with $n$.
In this article, we follow their approach.

\noindent Denote $d(f,\Sigma) = \inf_{g \in \Sigma} \Vert f-g \Vert_{\infty}$ the distance which derives from the sup-norm. For $\rho_n >0$, define
$$\tilde{\Sigma}(r,\rho_n) = \left\lbrace f \in \Sigma(r) \hspace{1 mm} \left\vert  \hspace{1 mm} \frac{}{} d(f,\Sigma(s)) \geq \rho_n \right. \right\rbrace.$$
Consider the model
\begin{equation*}
\mathcal{P}_n = \Sigma(s) \cup \tilde{\Sigma}(r,\rho_n),
\end{equation*}

\noindent for $0 < r < s$.

\begin{definition}[Honesty]
 The confidence band $C_n$ is called asymptotically honest with level

\hspace{-7 mm} $0<\alpha<1$ for 
$\mathcal{P}_n$ if it satisfies the asymptotic coverage inequality
$$
\liminf_{n} \inf_{f \in \mathcal{P}_n} P_f(f \in C_n) \geq 1 - \alpha.
$$ 
\end{definition}

\begin{definition}[Adaptivity]
The confidence band $C_n$ is called adaptive over $\mathcal{P}_n$ if there exists a constant $L$ such that for every $n \in \mathbb{N}$,

$$
\sup_{f \in \Sigma(s)} E_f|C_n| \leq L r_n(s), \hspace{1,0 cm} \sup_{f \in \tilde{\Sigma}(r,\rho_n)} E_f|C_n| \leq L r_n(r).
$$
\end{definition}

One is in turn interested in finding the smallest possible seqence $\rho_n$.
Hoffmann and Nickl provide a lower bound for $\rho_n$ which is sharp as they are able to
construct confidence bands for $\rho_n$ of the order of this lower bound.
The present article adapts their proofs to the regression model. \newline
To compute the lower bound, Hoffmann and Nickl reduce to a testing problem whose minimax rate
of testing is $\rho_n$, and we proceed likewise. To construct confidence bands, we use
adaptive estimators, which existence we prove, as well as a concentration inequality for a certain gaussian process. \newline
We first give the main result, and carry on with
its proof. It relies on some additional results, which are discussed in a subsequent section.

\section{Existence of adaptive and honest confidence bands}
\subsection{The main result}

We may now give a precise statement of the existence of confidence bands. 

\begin{theorem}
\label{thm}

Let $s > r > 0$ and $B>0$ be given.

\begin{itemize}
\item Assume $ r > 1/2$ and let $\alpha < 1/2$. Suppose $C_n$ is a confidence band that is honest with level $\alpha$ over $\mathcal{P}_n$, and adaptive.
Then necessarily, 
\begin{equation}
\label{thm1}
\liminf_{n} \frac{\rho_n}{r_n(r)} > 0.
\end{equation}

 \item Let $0 < \alpha < 1$. There exists a sequence $\rho_n$ satisfying
\begin{equation}
\label{thm2}
\limsup_{n} \frac{\rho_n}{r_n(r)} < \infty
\end{equation}
and a confidence band $C_n$ that is honest with level $\alpha$ over $\mathcal{P}_n$, and adaptive. 
\end{itemize}
\end{theorem}

The first part of the theorem means that the crown removed around $\Sigma(s)$ has to be large enough. Indeed,
there are functions in $\Sigma(r)$ we cannot distinguish (as Hoffmann and Nickl \cite{HN} explain) from those of $\Sigma(s)$, and they must thus be excluded from the model,
as detailed in the next section.
As explained in \cite{HN}, the sets constructed cannot be easily computed.
In the next two sections we prove this theorem, using auxiliary results we discuss later.

\subsection{Proof of the lower bound}
\label{lb}
In all what follows, $(j_n^*)_{n \geq 0}$ is a positive sequence such that $$2^{-j_n^*r} \simeq \left( \frac{\log n}{n} \right)^{\frac{r}{2r+1}},$$
where $\simeq$ means that the two sequences involved are dominated by each other. Before starting the proof,
we define a certain set we use in it, then precise the notations we use for tests. \newline

For $j$ big enough and any $m$, $\psi_{jm}$ is supported in the interior of $[0,1]$.
Furthermore, since $\psi$ has a compact support, denote $[a,b] \supset supp(\psi)$ and $c_0^{-1} = \lceil b-a \rceil$.
Now, the

\begin{equation*}
\psi_{j,mc_0^{-1}}, \, m = 1, \, \dots, \, c_0(2^j-1)
\end{equation*}

\noindent have disjoint supports. Note that $c_0$ depends only on $\psi$. \newline
Let $(j_n)_{n \geq 0}$ be a positive sequence tending to $\infty$. For all $ 1 \leq m \leq c_0(2^{j_n} -1)$, let

\begin{equation*}
f_m = 2^{-j_n(r+1/2)} \psi_{j_n,c_0^{-1}m}.
\end{equation*}

\noindent ($f_m$ depends on $n$ but since it is apparent that the set $\mathcal{M}_n$ depends on $n$,
we do not repeat it for $f_m$ so as to simplify notations.) Define

\begin{equation*}
\mathcal{M}_n(j_n) = \{f_1, ..., f_{c_0(2^{j_n}-1)} \}.
\end{equation*}

$\mathcal{M}_n(j_n)$ is a subset of $\Sigma(r)$ and, under suitable choices of
$j_n$ and $\rho_n$ precised below, it is in fact a subset of $\tilde{\Sigma}(r,\rho_n)$.\newline

The existence of confidence bands for $\mathcal{P}_n$ is related to the possibility of testing accurately the alternative~:
\begin{equation}
\label{alt}
H_0 : f = 0 \hspace{3mm} \text{against} \hspace{3mm} H_1 : f \in \mathcal{M}_n(j_n).
\end{equation}

A test $T_n$ based on a sample of size $n$ is any
$$ T_n : \Omega \rightarrow \{0,1\}$$ which is measurable with respect to $\sigma(Y_1, ..., Y_n)$, the $\sigma$-algebra generated
by $Y_1, ..., Y_n$. We denote by $\mathcal{T}_n$ the set of all tests $T_n$. \newline

To assess the quality of the test $T_n$ designed to solve the testing problem~\ref{alt}, we use the sum of the errors of first and second type :
$$ r(T_n,j_n) = E_{f_0}(T_n) + \sup_{f \in \mathcal{M}_n(j_n)} E_{f}(1-T_n),$$ where $f_0 = 0$ identically. \newline

We need the two following results.
The first one shows that for $2^{-j_nr} \ll (\frac{\log n}{n})^{\frac{r}{2r+1}}$, we cannot test efficiently $f_0$ againt $\mathcal{M}_n(j_n)$.
Now, $2^{-j_nr}$ is intuitively the distance in the sup-norm
between $\mathcal{M}_n(j_n)$ and $\Sigma(s)$. As a result, Proposition~\ref{prop1} means that close to $\Sigma(s)$, we find sets which cannot
be statistically distinguished from it. The proof is in section \ref{proofprop1}.

\begin{proposition}
\label{prop1}
Let $(j_n)_{n \geq 0}$ be any sequence satisfying $2^{-j_nr} = o(2^{-j_n^*})$ as $n \rightarrow \infty$.
Then $$\liminf_{n \rightarrow \infty} \inf_{T_n \in \mathcal{T}_n} r(T_n,j_n) \geq 1.$$
\end{proposition}

Once we are given a confidence band, we can construct an obvious test to decide the testing problem~\ref{alt}.
Let $$T_n^{0} = 1_{C_n \cap \mathcal{M}_n(j_n) \neq \emptyset}.$$ We accept $f_0$ if no $f_m$ is in $C_n$ and reject otherwise.
If $j_n$ satisfies certain conditions, $T_n^0$ is a relevant test for the testing problem~\ref{alt}.
The problem is that it may be possible to find $j_n$ such that both these conditions are satisfied and $2^{-j_nr} \ll (\frac{\log n}{n})^{\frac{r}{2r+1}}$, leading to
a contradiction.
Lemma~\ref{triple} indeed shows that if $\rho_n$ goes to $0$ too quickly,
we may find $j_n$ such that both $r(T_n^*,j_n) \to 0$ and the conditions of Proposition~\ref{prop1} are satisfied.
To remedy this, the exclusion zone around $\Sigma(s)$ must contain all functions which prevent adaptation, and therefore
$\rho_n$ must be big enough. \newline

\begin{lemma}
\label{triple}
Assume $\lim_{n \to \infty} \frac{\rho_n}{r_n(r)} = 0$. Then there exists $j_n$ such that
\begin{itemize}

\item $\mathcal{M}_n(j_n) \subset \tilde{\Sigma}(r,\rho_n)$, for all $n$ large enough, depending only on $B,r$ and $s$.

\item $r_n(s)/2^{-j_nr} \to 0$ as $n \to \infty$.

\item $2^{-j_nr} = o(2^{-j_n^*r})$ as $n \to \infty$.
\end{itemize} 

\end{lemma}

\begin{proof}

Let $j_n = \min(\lfloor - \frac{\log 2 \rho_n}{r \log 2} \rfloor - 1,\lfloor - \frac{ r_n(s)\log n}{r \log 2} \rfloor)$. Then 
$$
\max (2\rho_n,r_n(s)\log n) \leq 2^{-j_nr}.
$$

First, for all $1 \leq m \leq 2^{j_n} - 1$,

\begin{equation*}
d(f_m,\Sigma(s)) \geq  2^{-j_nr}(1- B 2^{j_n(r-s)})
\end{equation*}

\noindent (as in Hoffmann and Nickl, in section 2.1, inequality (2.5)).
Since $\rho_n$ is less than the right-hand side for all $n$ large enough depending only on $B$, $r$ and $s$,
$\mathcal{M}_n(j_n)$ is included in $\tilde{\Sigma}(r,\rho_n)$
for all $n$ large enough. Then,

\begin{equation*}
\frac{r_n(s)}{2^{-j_nr}} \leq \frac{r_n(s)}{r_n(s) \log n} \to 0
\end{equation*}

as $n \to \infty$.

Finally, write $j_n = \min (\lfloor a_n \rfloor, \lfloor b_n \rfloor)$ ; $a_n,b_n \to \infty$ as $n \to \infty$.
So, $\lfloor a_n \rfloor \sim a_n$, but

\begin{equation*}
2^{-a_nr} = o(r_n(r)) \hspace{3 mm} \text{and} \hspace{3 mm} 2^{- \lfloor a_n \rfloor r} \leq 2^{- a_n r}2^r \hspace{3 mm} \text{so} \hspace{3 mm}
2^{- \lfloor a_n \rfloor r} = o(2^{-j_n^*r}).
\end{equation*}

Likewise, $2^{- \lfloor b_n \rfloor r} = o(2^{-j_n^*r})$. So $2^{-j_nr} = o(2^{-j_n^*r})$.

\end{proof}

We now prove Proposition~\ref{prop2}, which is equivalent to the first part of Theorem~\ref{thm1}. We proceed by \emph{reductio ad absurdum}. The main thing is to
prove $T_n^{0}$ is consistent for the testing problem~\ref{alt}.

\begin{proposition}
\label{prop2}
Assume $\liminf_{n \to \infty} \frac{\rho_n}{r_n(r)} = 0$. Then we cannot find an honest and adaptive confidence band
of level $\alpha<1/2$ for $\mathcal{P}_n$.
\end{proposition}

\begin{proof}

Assume otherwise. Take a subsequence of $\frac{\rho_n}{r_n(r)}$ which tends towards $0$, and still note $n$ the subscript. Take $j_n$ as in Lemma~\ref{triple}. We show first that $r(T_n^*,j_n) \to 0$ as $n \to \infty$. 
\begin{eqnarray*}
E_{f_0}(T_n^0) &=& P_{f_0}(C_n \cap \mathcal{M}(j_n) \neq \emptyset)
\\ &=& P_{f_0} \left( \frac{}{} C_n \cap \mathcal{M}_n(j_n) \neq \emptyset, f_0 \in C_n \right) + P_{f_0}\left( \frac{}{} C_n \cap \mathcal{M}_n(j_n) \neq \emptyset, f_0 \notin C_n \right)
\\ &\leq & P_{f_0}\left( \frac{}{} |C_n| \geq d(\Sigma(s),\mathcal{M}_n(j_n)) \right) + P(f_0 \notin C_n)
\\ &\leq & \frac{E_{f_0}(|C_n|)}{d(\Sigma(s),\mathcal{M}_n(j_n))} + \alpha + o(1)
\end{eqnarray*}
thanks to Markov's inequality and honesty of $C_n$.

So 
$$
E_{f_0} (T_n^0) \leq \frac{Lr_n(s)}{2^{-j_nr}(1-B} 2^{j_n(r-s)}) + \alpha + o(1) = \alpha + o(1)
$$
thanks to Lemma~\ref{triple} (second point) and adaptivity of $C_n$.

For $f_m \in \mathcal{M}_n(j_n)$, 

\begin{eqnarray*}
 E_{f_m}(1-T_n^0) &=& P_{f_m}(C_n \cap \mathcal{M}_n(j_n) = \emptyset)
\\ &\leq &  P_{f_m}(f_m \notin C_n)
\\ &\leq & \sup_{f \in \mathcal{P}_n} P_f(f \notin C_n)
\end{eqnarray*}
thanks to Lemma~\ref{triple} (first point). So

\begin{equation*}
\sup_{f_m \in \mathcal{M}_n(j_n)} E_{f_m}(T_n^0) \leq \sup_{f \in \mathcal{P}_n} P_f(f \notin C_n) \leq \alpha + o(1)
\end{equation*}

\noindent by honesty of $C_n$.
Therefore we have
\begin{equation*}
\limsup_{n \to \infty} \, r(T_n^0,j_n) \leq 2\alpha < 1
\end{equation*}

\noindent but

\begin{equation*}
\liminf_{n \to \infty} \inf_{T_n \in \mathcal{T}_n} r(T_n,j_n) \geq 1,
\end{equation*}

\noindent thanks to Lemma~\ref{triple} (third point) and Proposition~\ref{prop1}, which is a contradiction. 

\end{proof}

\subsection{Proof of upper bound}
\label{ub}
\subsubsection{Construction of $C_n$}

Let $\rho_n = \lambda r_n(r)$, with $\lambda$ chosen below.
We will estimate the function $f$ to construct confidence bands. We need two different estimators.
We use the local polynomial estimator of order $l$ with $h = \left( \frac{\log n}{n}\right)^{\frac{1}{2r+1}}$, $f_n(h)$, which we note $f_n$ in this section, defined in Theorem~\ref{lpe}.
Thanks to Equation~\ref{bias}, we know that for some $b>0$, for all $f \in \Sigma(r)$,
\begin{equation}
\label{dn}
\Vert E_f f_n - f \Vert_{\infty} \leq b r_n(r) 
\end{equation}
$\hat{f}_n$ is the estimator given by Theorem~\ref{thm5}. $\hat{f}_n$ is adaptive over $\Sigma(s) \cup \Sigma(r)$ with rate $r_n(\cdot)$. In other words, for $D = D(B,l)$,

\begin{equation}
\sup_{f \in \Sigma(t)} E_f \Vert \hat{f}_n - f \Vert_{\infty} \leq Dr_n(t) \hspace{1.0 cm} t = r,s. 
\end{equation}
We are now able to construct the confidence band $C_n$. The idea is to adapt the width of $C_n$ with the likelihood that $f$ be in $\Sigma(r)$ or in $\Sigma(s)$. Define $$d_n = d(f_n,\Sigma(s)).$$ 
If $d_n$ is greater than some constant times $r_n(r)$, Equation~\ref{dn} shows that it is unlikely that $f$ is in $\Sigma(s)$
(intuitively, $f_n \approx E_f f_n$), and thus we choose $\hat{f}_n \pm Lr_n(r)$. Elsewhere, we choose $\hat{f}_n \pm Lr_n(s)$. \newline
Accordingly, define $C_n$ :
\begin{equation}
\label{cn}
\hat{f}_n \pm Lr_n(r) \hspace{0.3 cm} \text{if } d_n > \tau \hspace{0.3 cm} \text{and} \hspace{0.3 cm} \hat{f} \pm Lr_n(s) \hspace{0.3 cm} \text{if } d_n \leq \tau,
\end{equation}
where $\tau = \kappa r_n(r)$ and where $\kappa$ and $L$ are constants chosen below. Let us now prove that $C_n$ is honest and adaptive.

\subsubsection{Proof of honesty and adaptivity}

We first prove that $C_n$ is an honest confidence band for $\mathcal{P}_n$. Let $0 < \alpha < 1$. If $f \in \Sigma(s)$ we have, using adaptivity of $\hat{f}_n$ and Markov's inequality,

\begin{eqnarray*}
\inf_{f \in \Sigma(s)} P_f(f \in C_n) &\geq & 1 - \sup_{f \in \Sigma(s)} P_f \left(\Vert \hat{f}_n - f \Vert_{\infty} > L r_n(s) \right) \\
 & \geq &  1 - \frac{1}{Lr_n(s)} \sup_{f \in \Sigma(s)} E_f \Vert \hat{f}_n - f \Vert_{\infty} \\
 & \geq & 1 - \frac{D}{L}
\end{eqnarray*}
which is greater than $1 - \alpha$ for $L$ large enough. 

Now, if $f \in \tilde{\Sigma} (r,\rho_n)$, we have, using again Markov's inequality,
$$
\inf_{f \in \tilde{\Sigma}(r,\rho_n)} P_f(f \in C_n) \geq 1 - \frac{\sup_{f \in \tilde{\Sigma}(r,\rho_n)} E_f \Vert \hat{f}_n - f \Vert_{\infty}}{Lr_n(r)} - \sup_{f \in \tilde{\Sigma}(r,\rho_n)} P_f(d_n \leq \tau)
$$
and the first substracted term is smaller than $\alpha/2$ for $L$ large enough. Now,  $P_f(d_n \leq \tau)$ equals, for every $f \in \tilde{\Sigma}(r,\rho_n)$,

\begin{eqnarray*}
P_f \left( \inf_{g \in \Sigma(s)} \Vert f_n - g \Vert_{\infty} \leq \kappa r_n(r) \right) 
& \leq & P_f \left( \frac{}{} \inf_g \Vert f-g \Vert_{\infty} - \Vert f_n - E_f f_n \Vert_{\infty} - \Vert E_f f_n - f \Vert_{\infty} \leq \kappa r_n(r) \right) \\
& \leq & P_f \left( \frac{}{} \rho_n - \Vert E_f f_n - f \Vert_{\infty} - \kappa r_n(r) \leq \Vert f_n - E_f f_n \Vert_{\infty} \right) \\
& \leq & P_f\left( \frac{}{} \Vert f_n - E_f f_n \Vert_{\infty} \geq (\lambda - \kappa - b)r_n(r) \right) \\
& \leq & c_1 n^{-\gamma_1} = o(1)
\end{eqnarray*}
thanks to Corollary~\ref{kerk}, by choosing $\lambda$ large enough. This completes the proof of honesty of the band. Let us now deal with adaptivity.

By definition of $C_n$ we have
$$
|C_n| \leq Lr_n(r)
$$
so the case $f \in \tilde{\Sigma}(r,\rho_n)$ is immediate. If $f \in \Sigma(s)$ then,

\begin{eqnarray*}
E_f |C_n| &\leq & Lr_n(r) P_f(d_n > \tau) + Lr_n(s) \\
& \leq & L r_n(r) P_f \left( \inf_{g \in \Sigma(s)} \Vert f_n - g \Vert_{\infty} > \kappa r_n(r) \right) + Lr_n(s) \\
& \leq & Lr_n(r) P_f \left( \frac{}{} \Vert f_n - f \Vert_{\infty} > \kappa r_n(r) \right) + Lr_n(s) \\
& \leq & Lr_n(r) P_f \left( \frac{}{} \Vert f_n - E_f f_n \Vert_{\infty} > (\kappa - b) r_n(r) \right) + Lr_n(s) \\
& \leq & Lr_n(r)c_2 n^{- \gamma_2} + Lr_n(s) = O(r_n(s))
\end{eqnarray*}
thanks to Corollary~\ref{kerk}, since $\gamma_2$ is sufficiently large if $\kappa$ is chosen large enough.

\section{Auxiliary results}

The first section proves Proposition~\ref{prop1}. It is the adaptation of the proof in Lepski and Tsybakov~\cite{Lepski}. Lemma~\ref{stand} is a standard procedure to
lower bound the error ; Lemma~\ref{lemxi} and Lemma~\ref{lemalpha} deal with the functions
in $\mathcal{M}_n(j_n)$, and Lemma~\ref{end} uses them all to give the required bound, thanks to Bahr-Esseen's inequality~\cite{BE}.
The second section presents the estimators used to contruct the confidence bands and to prove the existence of adaptive estimators.
The third section deals with the concentration inequality used in proving the bands are indeed honest and adaptive, and that the estimator we construct
is indeed adaptive.

\subsection{Proof of Proposition~\ref{prop1}}
\label{proofprop1}
Throughout this section, $T_n$ is any test, $(j_n)_{n \geq 0}$ is a non negative real sequence such that $2^{-j_nr} = o(2^{-j_n^*r})$ as $n \rightarrow \infty$, and $M = c_0(2^{j_n} - 1)$. Recall the definition of
$\mathcal{M}_n(j_n)$ and that $f_0 = 0$.

\begin{lemma}
\label{stand}
For all $0 < \eta < 1$,
$$r(T_n,j_n) \geq (1 - \eta) P\left(\frac{1}{M} \sum_{m = 1}^M \xi_m \geq 1 - \eta \right),$$
where $\xi_m = \frac{dP_m}{dP_0}(\varepsilon_1, ..., \varepsilon_n)$ and $P_m = P_{f_m}$.
\end{lemma}

\begin{proof}
\begin{eqnarray*}
r(T_n,j_n) &=& E_{f_0}(T_n) + \sup_{f \in \mathcal{M}_n(j_n)} E_{f}(1-T_n)
\\ &\geq & E_{f_0}(T_n) + \frac{1}{M} \sum_{m=1}^M E_{f_m}(1-T_n)
\\ &\geq &  E_{f_0}(T_n) + E_{f_0}((1-T_n)Z_n)
\end{eqnarray*}
where $$Z_n = \frac{1}{M} \sum_{m = 1}^M \frac{dP_m}{dP_0}(Y_1, ..., Y_n).$$
Let $0 < \eta < 1$.

\begin{eqnarray*}
\text{If } T_n = 1 &\text{then}& T_n + (1-T_n)Z_n = 1 \geq 1-\eta \\
&\text{else}& T_n + (1-T_n)Z_n = Z_n. 
\end{eqnarray*}
So,

\begin{eqnarray*}
T_n + (1-T_n) Z_n & \geq & (T_n + (1-T_n)Z_n)\boldmath{1}_{Z_n \geq 1-\eta} \\
                  & \geq & (1-\eta)\boldmath{1}_{Z_n \geq 1-\eta}. \\
\end{eqnarray*}
 As a result,
 
\begin{eqnarray*}
r_n(T_n,j_n) & \geq & (1-\eta)P_{f_0}(Z_n \geq 1-\eta)
\\ &=& (1- \eta)P\left(\frac{1}{M} \sum_{m = 1}^M \frac{dP_m}{dP_0}(\varepsilon_1, ..., \varepsilon_n) \geq 1 - \eta\right).
\end{eqnarray*}
\end{proof}

We want to show that $P\left(\frac{1}{M} \sum_{m = 1}^M \xi_m \geq 1 - \eta\right) \rightarrow 1$ as $n \rightarrow \infty$. It will be sufficient to conclude
because this quantity is independent of $T_n$. To do this, note that
$$P\left(\frac{1}{M} \sum_{m = 1}^M \xi_m \geq 1 - \eta \right) = 1 - P\left(\frac{1}{M} \sum_{m = 1}^M \xi_m < 1 - \eta \right).$$
Now, 
\begin{eqnarray*}
P\left(\frac{1}{M} \sum_{m = 1}^M \xi_m < 1 - \eta \right) &=& P\left(\frac{1}{M} \sum_{m = 1}^M \xi_m -1 < -\eta \right) \\
      &\leq& P\left( \left| \frac{1}{M} \sum_{m = 1}^M \xi_m - 1 \right| > \eta \right) \\
\end{eqnarray*}
and we therefore want to prove that the last quantity tends towards $0$.

\begin{lemma}
\label{lemxi}
For $1 \leq m \leq M$, $$\xi_m = \exp\left(\frac{\alpha_m}{\sigma}\zeta_m - \frac{2\sigma^2}{\alpha_m^2}\right),$$
with $\zeta_m \sim \mathcal{N}(0,1)$ i.i.d. and $$\alpha_m^2 = \sum_{i=1}^n (f_m(x_i))^2.$$
(As for the $f_m$'s, we do not make explicit the dependence on $n$ of $\alpha_m$.)
\end{lemma}

\begin{proof}
\begin{eqnarray*}
\xi_m &=& \prod_{i=1}^n \exp\left(-\frac{(\varepsilon_i - f_m(x_i))^2 - \varepsilon_i^2}{2\sigma^2}\right)
\\ &=& \prod_{i=1}^n \exp\left(\frac{f_m(x_i)(2\varepsilon_i - f_m(x_i))}{2\sigma^2}\right)
\\ &=& \exp\left(\frac{1}{2\sigma^2} \sum_{i=1}^n f_m(x_i)(2\varepsilon_i - f_m(x_i))\right)
\\ &=& \exp\left(\frac{1}{\sigma^2} \sum_{i=1}^n f_m(x_i)\varepsilon_i - \frac{1}{2\sigma^2} \sum_{i=1}^n (f_m(x_i))^2\right). 
\end{eqnarray*}
Now, define $$\zeta_m = \frac{1}{\alpha_m \sigma} \sum_{i=1}^n f_m(x_i)\varepsilon_i.$$
Therefore, $$\xi_m = \exp(\frac{\alpha_m}{\sigma}\zeta_m - \frac{2\sigma^2}{\alpha_m^2}).$$
$\zeta_m$ is a gaussian random variable as a linear combination of independent gaussian random variables, and
straightforward computations show that $\zeta_m \sim \mathcal{N}(0,1)$.

To prove that the $\zeta_m$'s are independent, we first note that the vector $\zeta_1, ..., \zeta_M$ is a gaussian vector. Indeed, any linear combination
of its coordinates is a linear combination of the $\varepsilon_i$'s and consequently a gaussian random variable as above. Thus, it is sufficient to prove
that the covariances all equal $0$. \newline
But, for $m \neq m'$, $$Cov(\zeta_m,\zeta_{m'}) = (\sigma^2 \alpha_m \alpha_{m'})^{-1}\sum_{i=1}^n f_m(x_i)f_{m'}(x_i) = 0$$ since the $f_m$'s supports are not overlapping. 
\end{proof}

\begin{lemma}
\label{lemalpha}
For all $1 \leq m \leq M$, $$\alpha_m^2 = n2^{-j_n(2r+1)} + O(2^{j_n(1-2r)})$$
with the  last term independent of $m$.
\end{lemma}

\begin{proof}
We drop the $n$ subscript in $j_n$ so as to simplify notations. Recall $supp(\psi_{jm}) \subset [0,1]$. \newline
First,
\begin{eqnarray*}
\frac{1}{n} \sum_{i=1}^n 2^j \psi^2(2^jx_i-m) - \int_{\mathbb{R}}\psi^2(t) \, dt &=& \frac{1}{n} \sum_{i=1}^n 2^j \psi^2(2^jx_i-m) - 2^j\int_{[0,1]}\psi^2(2^jt-m) \, dt
\\ &=& 2^j\sum_{i=1}^n \int_{\frac{i-1}{n}}^{\frac{i}{n}}\left(\psi^2(2^jx_i-m) - \psi^2(2^jt-m)\right) \, dt
\\ &=& 2^j\sum_{i=1}^n \int_{\frac{i-1}{n}}^{\frac{i}{n}} \left( \frac{}{} \psi(2^jx_i-m) - \psi(2^jt-m) \right) \\  & & \hspace{3.0 cm} \left( \frac{}{} \psi(2^jx_i-m) + \psi(2^jt-m) \right) \, dt
\\ &\leq & 2^{2j} \frac{||\psi'||_{\infty}}{n}2||\psi||_{\infty} = K \frac{2^{2j}}{n}.
\end{eqnarray*} \newline

Here we use the mean value theorem and the fact that for all $i \in [1,n]$ and $t \in [\frac{i-1}{n},\frac{i}{n}]$, $\vert t - x_i \vert \leq \frac{1}{n}$.

Therefore, $$\sum_{i=1}^n 2^j \psi^2(2^jx_i-m) = n \int_{\mathbb{R}}\psi^2(t) \, dt + O(2^{2j})$$
and recall $\int_{\mathbb{R}}\psi^2(t) \, dt=1$. \newline

Now,
\begin{eqnarray*} 
\alpha_m^2 &=& 2^{-j(2r+1)} \sum_{i=1}^n \psi_{jm}^2(x_i)
\\ &=& 2^{-j(2r+1)} \sum_{i=1}^n 2^j \psi^2(2^jx_i-m)
\\ &=& 2^{-j(2r+1)} \left(n + O\left(2^{2j}\right) \right)
\\ &=& n2^{-j(2r+1)} + O\left(2^{j(1-2r)}\right).
\end{eqnarray*}

\end{proof}

\begin{lemma}
\label{end}
Let $0 < v < 1$.
$$P\left(\left|\frac{1}{M} \sum_{m = 1}^M \xi_m - 1\right| > \eta\right) \leq c_1 \frac{\exp(c_2 n2^{-j_n(2r+1)})}{(2^{j_n}-1)^v}$$
where $c_1$ and $c_2$ are positive constants.
\end{lemma}

\begin{proof}
First note that for all $1 \leq m \leq M$, $E(\xi_m) = 1$, so we have, by Markov's inequality 

\begin{eqnarray*}
P\left(\left|\frac{1}{M} \sum_{m = 1}^M \xi_m - 1\right| > \eta \right) &=& P\left(\left|\frac{1}{M} \sum_{m = 1}^M (\xi_m - E(\xi_m))\right| > \eta \right) \\ 
                &=& P\left(\left|\sum_{m = 1}^M (\xi_m - E(\xi_m))\right| > M\eta\right) \\
                &=& P\left(\left|\sum_{m = 1}^M (\xi_m - E(\xi_m))\right|^{1+v} > (M\eta)^{1+v}\right) \\
                & \leq & \frac{E\left(\left|\sum_{m=1}^M(\xi_m - E(\xi_m))\right|^{1+v}\right)}{(M\eta)^{1+v}}. \\
\end{eqnarray*}
Now, Bahr and Esseen recall that, for $1 \leq r \leq 2$ and $x,y$ complex numbers,
$ \vert x+y \vert^r + \vert x-y \vert^r \leq 2(\vert x \vert^r + \vert y \vert^r)$. This, and the inequality which bears their names give
\begin{eqnarray*}
E\left(\left|\sum_{m=1}^M(\xi_m - E(\xi_m))\right|^{1+v}\right) & \leq & \frac{K}{2}\sum_{m=1}^M E\left(\left|\xi_m-E(\xi_m)\right|^{1+v}\right) \\
             & \leq & K\sum_{m=1}^M E\left(|\xi_m|^{1+v} + 1 \right) \\
& \leq & K \left( \sum_{m=1}^M E\left(|\xi_m|^{1+v} \right) + M \right) 
\end{eqnarray*}
which leads us to

\begin{equation}
\label{eqbe}
P\left(\left|\frac{1}{M} \sum_{m = 1}^M \xi_m - 1\right| > \eta\right) \leq K \left( \frac{\sum_{m=1}^M E\left(|\xi_m|^{1+v}\right)}{(M\eta)^{1+v}} + \frac{1}{M^v\eta^{1+v}} \right).
\end{equation}

But, thanks to Lemma~\ref{lemalpha}, we have 
\begin{eqnarray*}
 E\left(|\xi_m|^{1+v}\right) &=& \exp\left(\frac{\alpha_m^2}{2 \sigma^2}(v+v^2)\right) \\
                  &=& \exp\left(\frac{1}{2 \sigma^2}(v+v^2)\left(n2^{-j_n(2r+1)} + O\left(2^{j_n(1-2r)}\right)\right)\right) \\
                  &=& \exp\left(\frac{1}{2 \sigma^2}(v+v^2)\left(n2^{-j_n(2r+1)}\right)\right) \exp\left(O\left(2^{j_n(1-2r)}\right)\right),\\
\end{eqnarray*}
with $\exp(O(2^{j_n(1-2r)})$ bounded independently of $m$ thanks to the same lemma by a constant $K'$, because $r>1/2$. \newline

Now, Equation~\ref{eqbe} leads to 
\begin{eqnarray*}
P\left(\left|\frac{1}{M} \sum_{m = 1}^M \xi_m - 1\right| > \eta\right) & \leq & K \left( \frac{MK'\exp\left(\frac{1}{2 \sigma^2}(v+v^2)\left(n2^{-j_n(2r+1)}\right)\right)}{(M\eta)^{1+v}} + \frac{1}{M^v\eta^{1+v}} \right)\\
&\leq & K \left( \frac{K'\exp\left(\frac{1}{2 \sigma^2}(v+v^2)\left(n2^{-j_n(2r+1)}\right)\right)}{M^v\eta^{1+v}} + \frac{1}{M^v\eta^{1+v}} \right)\\          
& \leq & K \max(K',1) \frac{\exp\left(\frac{1}{2 \sigma^2}(v+v^2)\left(n2^{-j_n(2r+1)}\right)\right)}{c_0^v(2^{j_n}-1)^v\eta^{1+v}} \\
\end{eqnarray*}
since $1 \leq \exp \left(\frac{1}{2 \sigma^2}(v+v^2)\left(n2^{-j_n(2r+1)} \right) \right)$.
\end{proof}

As a consequence, we just have to plug in the ``value'' of $j_n$ in
$$P\left(\left|\frac{1}{M} \sum_{m = 1}^M \xi_m - 1\right| > \eta\right) \leq c_1 \frac{\exp\left(c_2 n2^{-j_n(2r+1)}\right)}{(2^{j_n}-1)^v}.$$
Now, $2^{-j_nr} = o(2^{-j_n^*r})$ as $n \rightarrow \infty$.
But $2^{-j_n^*r} \simeq (\frac{\log n}{n})^{\frac{r}{2r+1}}$ so $2^{-j_n(2r+1)} = o(\frac{\log n}{n})$ as $n \rightarrow \infty$ and
therefore the argument of the exponential is negligible in front of $\log n$. Since $2^{j_n} \geq (\frac{n}{\log n})^{\frac{1}{2r+1}}$, the
ratio tends towards $0$ as $n \rightarrow \infty$. The discussion before Lemma~\ref{lemxi} shows this completes the proof of Proposition~\ref{prop1}.

\subsection{Estimators}

\subsubsection{Local polynomial estimator}

We fix an integer $l \geq s$. $t$ is any number such that $0 < t \leq l$. In this section, $E$ stands for $E_f$. 
Reading through the proofs of Proposition~1.13 and Theorem 1.8 in Tsybakov \cite{Tsyba}, we get the following.

\begin{theorem}
\label{lpe}
There exist $n_0 \in \mathbb{N}$, $c_1 > 0$ such that,
for all $n \geq n_0$, for all $\left(\log n/n \right)^{1/(2l+1)} \geq h \geq \log n/n$,
there exist functions $x \in [0,1] \mapsto W_{ni}(h,x)$, $1 \leq i \leq n$ verifying :
\begin{itemize}
\item $\sup_{i,x} |W_{ni}(x)| \leq \frac{c_1}{nh}$
\item $\sum_{i=1}^{n} |W_{ni}(x)| \leq c_1$
\item $W_{ni}(x) = 0 \hspace{3 mm} \text{if} \hspace{3 mm} |x_i - x| > h$
\item $x \mapsto W_{ni}(x)$ is continuous.
\end{itemize}
There exist further $c_2 > 0$ such that,
for all $0 < t \leq l$, for all $f \in \Sigma(t)$, the local polynomial estimator of order $l$,
\begin{equation}
\label{fn*}
f_n(h)(x) = \sum_{i=1}^{n} W_{ni}(h,x) Y_i = \sum_{i=1}^{n} W_{ni}(x) Y_i
\end{equation}
satisfies :
\begin{eqnarray}
\label{bias}
\Vert E f_n(h) - f \Vert_{\infty} \leq \frac{c_1\Vert f \Vert_{t,\infty}}{l!} h^t = B(h,f) \\
\label{variance}
E \Vert f_n(h) - E f_n(h) \Vert_{\infty}^2 \leq c_2^2 \frac{\log n}{nh} = c_2^2\sigma^2(h,n).
\end{eqnarray}
\end{theorem}

For $h = h_n(t) = \left( \frac{\log n}{n}\right)^{\frac{1}{2t+1}}$, this implies that $f_n(h)$ satisfies

\begin{equation*}
\sup_{f \in \Sigma(t)} E_f\Vert f_n(h_n(t)) - f \Vert_{\infty} \leq D \left( \frac{\log n}{n}\right)^{\frac{t}{2t+1}},
\end{equation*}

\noindent with $D$ depending on $B$ and $l$. This means that $f_n(h_n(t))$ is rate optimal over $\Sigma(t)$. Thanks to it, we shall now
construct adaptive estimators, rate optimal over a range of regularity indexes, by proving Theorem~\ref{thm5}.

\subsubsection{Proof of Theorem~\ref{thm5} : existence of adaptive estimators}

\label{pth5}

The proof is an adaptation of that of Giné and Nickl \cite{GN2} in the density case, which uses Lepski's method.
For each $0 < t \leq l$, we have an estimator $f_n(h_n(t))$ rate optimal over $\Sigma(t)$. We want to devise a procedure which allows us
to choose $h_n$ according only to the data, such that if $f \in \Sigma(t)$, $h$ will be roughly $h(t)$.
Lepski's method consists in discretising the set of all possible bandwiths, and choosing
$h_n$ so that, for all $h \simeq h_n(t) < h_n$, the distance between $f_n(h)$ and $f_n(h_n)$ is of the order
the optimal rate for $\Sigma(t)$, $r_n(t)$. \newline

Fix $\rho > 1$. Define

\begin{equation*}
\mathcal{H} = \left\lbrace \rho^{-k} \, \left\vert \, k \geq 0, \, \rho^{-k} > \frac{(\log n)^2}{n} \right. \right\rbrace
\end{equation*}
and
\begin{equation}
\label{ineqm}
\hat{h}_n = \max \left\lbrace h \in \mathcal{H} \, \left\vert \,
\forall g < h, \, \Vert f_n(\hat{h}_n ) - f_n(g) \Vert_{\infty} \leq \left( M \frac{\log n}{ng} \right)^{1/2} \right. \right\rbrace
\end{equation}
with $M = 16(\sqrt{2\sigma c_1K} + c_2)^2$ ; $K$ is choosen later. $|\mathcal{H}| \simeq \log n$. \newline
Indeed, for all $t>0$, $\frac{\log n}{n} \leq h_n(t) \leq 1$ and we choose $(\log n)^2$ for practical reasons (this choice is not restrictive since
indeed, what matters in this ratio is the power of $n$, not that of the logarithm).

The adaptive estimator is $\hat{f}_n = f_n(\hat{h}_n )$ but we will write $f_n(\hat{h}_n )$ in the proof.
Fix now $t>0$ and $f \in \Sigma(t)$. Define
\begin{equation*}
h_f = \max \left\lbrace h \in \mathcal{H} \left\vert \, B(h,f) \leq \frac{\sqrt{M}}{4} \sigma(h,n) \right.  \right\rbrace,
\end{equation*}
which verifies, thanks to Theorem~\ref{lpe},
\begin{equation*}
h_f \simeq \left( \frac{\log n}{n} \right)^{\frac{1}{2t+1}}.
\end{equation*}

We now bound $ E \Vert f_n(\hat{h}_n) - f \Vert_{\infty}$. We distinguish the cases $\lbrace \hat{h}_n \geq h_f \rbrace$ and $\lbrace \hat{h}_n < h_f \rbrace$.

First,
\begin{eqnarray*}
& & E \Vert \frac{}{} f_n(\hat{h}_n) - f \Vert_{\infty} I_{\hat{h}_n \geq h_f} \\
&\leq & E \left( \Vert f_n(\hat{h}_n) - f_n(h_f) \Vert_{\infty} + \frac{}{} \Vert f_n(h_f) - E f_n(h_f) \Vert_{\infty}
+ \Vert E f_n(h_f) - f \Vert_{\infty} \right) I_{\hat{h}_n \geq h_f} \\
&\leq & E \left( \sum_{h \geq h_f} \Vert f_n(h) - f_n(h_f) \Vert_{\infty} I_{\hat{h}_n = h}
\frac{}{} \right)
+ E \Vert f_n(h_f) - E f_n(h_f) \Vert_{\infty}
+ \Vert E f_n(h_f) - f \Vert_{\infty} \\
&\leq & \sqrt{M} \sigma(h_f,n) P(\hat{h}_n \geq h_f)
+ E \Vert f_n(h_f) - E f_n(h_f) \Vert_{\infty}
+ \Vert E f_n(h_f) - f \Vert_{\infty} \\
&\leq & \sqrt{M} \sigma(h_f,n) + c_2 \sigma(h_f,n) + \frac{\sqrt{M}}{4} \sigma(h_f,n) = O(\sigma(h_f,n))
\end{eqnarray*}

Then,
\begin{eqnarray*}
& & E \Vert f_n(\hat{h}_n ) - f \Vert_{\infty} I_{\hat{h}_n < h_f} \\
&= & \sum_{h \in \mathcal{H}, h < h_f} E \left( \frac{}{} \Vert f_n(\hat{h}_n) - f \Vert_{\infty} I_{\hat{h}_n = h} \right) \\
&\leq & \sum_{h \in \mathcal{H}, h < h_f} E \left( \frac{}{} \Vert f_n(\hat{h}_n) - Ef_n(\hat{h}_n) \Vert_{\infty}
+ \Vert Ef_n(\hat{h}_n) - f \Vert_{\infty} \right) I_{\hat{h}_n = h} \\
&\leq & \sum_{h \in \mathcal{H}, h < h_f} \left( \frac{}{} E \left( \frac{}{} \Vert f_n(\hat{h}_n) - Ef_n(\hat{h}_n) \Vert_{\infty} I_{\hat{h}_n = h} \right)
+ \Vert Ef_n(\hat{h}_n) - f \Vert_{\infty} E \left(I_{\hat{h}_n = h} \right) \right) \\
&\leq & \sum_{h \in \mathcal{H}, h < h_f}  \left( \frac{}{} c_2 \sigma \left(\frac{(\log n)^2}{n},f \right) P\left( \hat{h}_n = h \right)^{1/2} \right) + B(h_f,f)P(\hat{h}_n < h_f)\\
&\leq & \frac{c_2}{\sqrt{\log n}} \sum_{h \in \mathcal{H}, h < h_f} P\left( \hat{h}_n = h \right)^{1/2} + B(h_f,f).\\
\end{eqnarray*}

Now, for $h< h_f$, and writing $h^+ = \rho h$,
\begin{eqnarray*}
 P\left( \hat{h}_n = h \right) & \leq & P \left( \text{there is $g \leq h$, such that $h^+$ and $g$ do not satisfy the inequality}~\ref{ineqm} \right) \\
& \leq & \sum_{g \leq h, g \in \mathcal{H}} P \left( \frac{}{} \Vert f_n(h^+) - f_n(g) \Vert_{\infty}
\geq \sqrt{M} \left(\frac{\log n}{ng} \right)^{1/2} \right) \\
& \leq & \sum_{g \leq h, g \in \mathcal{H}} P \left( \frac{}{} \Vert f_n(h^+) - f_n(g) \Vert_{\infty} \geq \sqrt{M} \sigma(g,n) \right) \\
\end{eqnarray*}

For $g \leq h$, we have

\begin{eqnarray*}
& & \Vert f_n(h^+) - f_n(g) \Vert_{\infty} \\
& \leq &  \Vert f_n(h^+) - Ef_n(h^+) \Vert_{\infty} + \Vert Ef_n(h^+) - f \Vert_{\infty} \\
& & + \Vert f - Ef_n(g) \Vert_{\infty} + \Vert Ef_n(g) - f_n(g) \Vert_{\infty} \\ \vspace{1 mm}
&\leq& 2 B(h_f,f) + \Vert f_n(h^+) - Ef_n(h^+) \Vert_{\infty} + \Vert Ef_n(g) - f_n(g) \Vert_{\infty} \\ 
\end{eqnarray*}
since $B(h,f)$ increases with $h$,
\begin{eqnarray*}
&\leq& 2 \frac{\sqrt{M}}{4} \sigma(h_f,n) + \Vert f_n(h^+) - Ef_n(h^+) \Vert_{\infty}
+ \Vert Ef_n(g) - f_n(g) \Vert_{\infty} \\
&\leq& \frac{\sqrt{M}}{2} \sigma(g,n)  + \Vert f_n(h^+) - Ef_n(h^+) \Vert_{\infty} + \Vert Ef_n(g) - f_n(g) \Vert_{\infty}
\end{eqnarray*}
since $\sigma(h,f)$ decreases with $h$,
so that,
\begin{eqnarray*}
&& P \left( \frac{}{} \Vert f_n(h^+) - f_n(g) \Vert_{\infty} \geq \sqrt{M} \sigma(g,n) \right) \\
&\leq& P \left( \frac{}{} \Vert f_n(h^+) - Ef_n(h^+) \Vert_{\infty} + \Vert Ef_n(g) - f_n(g) \Vert_{\infty}
\geq \sqrt{M} \sigma(g,n) - \frac{\sqrt{M}}{2} \sigma(g,n) \right) \\
&\leq& P\left( \frac{}{} \Vert f_n(h^+) - Ef_n(h^+) \Vert_{\infty} \geq  \frac{\sqrt{M}}{4} \sigma(h^+,n) \right)
+  P\left( \frac{}{} \Vert f_n(g) - Ef_n(g) \Vert_{\infty} \geq  \frac{\sqrt{M}}{4} \sigma(g,n) \right) \\
& \leq & 4 \exp(-K \log n)
\end{eqnarray*}
thanks to Proposition~\ref{prop5}. Therefore,

\begin{eqnarray*}
P(\hat{h}_n = h) &\leq& \sum_{g \leq h, g \in \mathcal{H}} P \left( \frac{}{} \Vert f_n(h^+) - f_n(g) \Vert_{\infty} \geq \sqrt{M} \sigma(g,n) \right) \\
& \leq & 4 \exp(-K \log n) \log n,
\end{eqnarray*}
and
\begin{eqnarray*}
& & E \Vert f_n(h_n) - f \Vert_{\infty} I_{h_n < h_f} \\
& \leq & \frac{c_2}{\sqrt{\log n}} \sum_{h \in \mathcal{H}, h < h_f} P\left(\hat{h}_n = h \right)^{1/2} + B(h_f,n) \\
& \leq & \frac{c_2}{\sqrt{\log n}} \left( \frac{}{} 4 \exp(-K \log n) \log n \right)^{1/2} \log n + B(h_f,n) \\
&=& O(\sigma(h_f,n))
\end{eqnarray*}
for $K$ large enough, choosen independently of $f$ or $t$, which completes the proof of Theorem~\ref{thm5}.

\subsection{A concentration inequality}

The following proposition is a key result which allows us to prove both the honesty and adaptivity of $C_n$ in Section~\ref{ub}
as well as the existence of the adaptive estimators.
Indeed, it gives an upper bound for the concentration of $f_n(h)$ around its expectation.
The proof relies on an inequality
due to Borell \cite{Bo}, which may be found in Theorem 1.7 in \cite{ledoux} in the form we use it.

\begin{theorem}
\label{Borell}
Let $G(x)$, $x \in T$, be a centered Gaussian process indexed by the countable set $T$, and such that $\sup_{x \in T} |G(x)| < \infty$ almost surely. Then $E \sup_{x \in T} |G(x)| < \infty$, and for every $r \geq 0$ we have 
$$
P \left\lbrace \left\vert \sup_{x \in T} |G(x)| - E \sup_{x \in T} |G(x)| \right\vert \geq r \right\rbrace \leq 2 e^{-r^2/2\sigma_0^2}
$$
where $\sigma_0^2 = sup_{x \in T} E(G^2(x)) < \infty$.
\end{theorem}
 
\begin{proposition}
\label{prop5}
Let $(\log n/n)^{1/(2l+1)} \geq h \geq \log n/n$ and $c_1$ and $c_2$ the constants in Theorem~\ref{lpe}. Let $f_n(h)$ be the local polynomial estimator of order $l$ for $f$.
Let $G_n = f_n(h) - E(f_n(h))$. Then,

\begin{equation}
\sigma_0^2 = E(G_n^2) \leq \frac{\sigma^2 c_1^2}{nh} \label{2} 
\end{equation}
\begin{equation}
\forall u \geq c_2 \left( \frac{\log n}{nh} \right)^{1/2},
P \left( \Vert G_n \Vert_{\infty} \geq u \right) \leq 2 \exp \left(-\frac{\log n}{2\sigma c_1}\left(\frac{u}{\left(\frac{\log n}{nh}\right)^{1/2}} - c_2\right)^2 \right). \label{3}
\end{equation}
\end{proposition}

\begin{proof}
Obviously, Equation~\ref{eqmd} and Equation~\ref{fn*} imply that $(G_n(x), \hspace{1 mm} 0 \leq x \leq 1)$,
which satisfies for all $ 0 \leq x \leq 1$,
\begin{eqnarray*}
G_n(x) &= &f_n(x) - E (f_n(x)) \\
&=& \sum_{i=1}^{n} W_{ni}(x)Y_i - E \left( \sum_{i=1}^{n} W_{ni}(x)Y_i \right) \\
&=& \sum_{i = 1}^n W_{ni}(x) \varepsilon_i 
\end{eqnarray*}
is a gaussian process.

\begin{itemize}

\item Recall the definition of $f_n$ and that the $\varepsilon_i$'s are i.i.d. $\mathcal{N}(0,\sigma^2)$. Then,
\begin{eqnarray*}
\sigma_0^2 &=& \sup_{x \in [0,1]} E(G_n^2(x)) \\
           &=& \sup_{x \in [0,1]} E\left(\sum_{i=1}^{n} W_{ni}(x) \varepsilon_i \right)^2 \\
           &=& \sup_{x \in [0,1]} \sum_{i=1}^{n} W_{ni}^2(x) \sigma^2 \\
           & \leq & \sigma^2 \Vert W_{ni} \Vert_{\infty} \sup_{x \in [0,1]} \sum_{i=1}^{n} \vert W_{ni}(x) \vert \\ 
           & \leq & \frac{\sigma^2 c_1^2}{nh}. \\
\end{eqnarray*}
\item Then, we want to bound the probability 

\begin{eqnarray*}
P \left(\Vert G_n \Vert_{\infty} \geq u \right) &=& P \left(\Vert G_n \Vert_{\infty} - E \Vert G_n \Vert_{\infty} \geq u - \Vert G_n \Vert_{\infty} \right) \\
& \leq & P \left(\left\vert \frac{}{} \Vert G_n \Vert_{\infty} - E \Vert G_n \Vert_{\infty} \right\vert \geq u - \Vert G_n \Vert_{\infty} \right) \\
&=& P \left(\left\vert \frac{}{} \sup_{x \in [0,1] \cap \mathbb{Q}} |G_n(x)| - E \sup_{x \in [0,1] \cap \mathbb{Q}} |G_n(x)| \right\vert \geq u - \Vert G_n \Vert_{\infty} \right) \\
\end{eqnarray*}
since $G_n$ is continuous.
Now, if $u - \Vert G_n \Vert_{\infty} \geq 0$, we can apply Theorem~\ref{Borell} to $(G_n(x), \hspace{1 mm} x \in [0,1] \cap \mathbb{Q})$ and write
\begin{equation*}
\label{Bor}
P \left(\Vert G_n \Vert_{\infty} \geq u \right) \leq 2 \exp \left( -\frac{(u-E \Vert G_n \Vert_{\infty})^2}{2\sigma_0^2} \right).
\end{equation*}

Finally, thanks to \ref{variance} and \ref{2}, for $u \geq c_2\left( \frac{\log n}{nh} \right)^{1/2}$,
\begin{equation*}
P \left( \Vert G_n \Vert_{\infty} \geq u \right) \leq 2 \exp \left(-\frac{\log n}{2\sigma c_1}\left(\frac{u}{\left(\frac{\log n}{nh}\right)^{1/2}} - c_2\right)^2 \right).
\end{equation*}
\end{itemize}
\end{proof}

In the proof of section~\ref{ub} we only need the following result, obtained with $h = \left( \frac{\log n}{n} \right)^{1/(2r+1)}$.

\begin{corollary}
\label{kerk}
Let $C \geq 0$.
Take $u = (c_2 + C) r_n(r)$, then
$$
P \left( \Vert G_n \Vert_{\infty} \geq u \right) \leq 2 n^{-C^2/2 \sigma c_1}.
$$
\end{corollary}

\textbf{Acknowledgements.} We would like to thank the Statistical Laboratory of Cambridge University which welcomed us for the
duration of our research internship. We are particularly indebted to
Richard Nickl who suggested the topic and made himself available every time we asked for advices,
so that we could write this note.

\end{document}